\documentclass{amsart}
\input epsf

\usepackage{amsfonts,amsthm,amsmath,amssymb,latexsym}
\usepackage{graphicx,color}
\usepackage[all]{xy}

\begin{document}

\author{Dragomir \v Sari\' c}
\thanks{This research was partially supported by National Science Foundation grant DMS 1102440.}

\address{Department of Mathematics, Queens College of CUNY,
65-30 Kissena Blvd., Flushing, NY 11367}
\email{Dragomir.Saric@qc.cuny.edu}

\address{Mathematics PhD. Program, The CUNY Graduate Center, 365 Fifth Avenue, New York, NY 10016-4309}

\theoremstyle{definition}

 \newtheorem{definition}{Definition}[section]
 \newtheorem{remark}[definition]{Remark}
 \newtheorem{example}[definition]{Example}

\newtheorem*{notation}{Notation}

\theoremstyle{plain}

 \newtheorem{proposition}[definition]{Proposition}
 \newtheorem{theorem}[definition]{Theorem}
 \newtheorem{corollary}[definition]{Corollary}
 \newtheorem{lemma}[definition]{Lemma}

\newcommand{\eps}{\varepsilon}
\newcommand{\G}{\Gamma}
\newcommand{\g}{\gamma}
\newcommand{\D}{\Delta}
\renewcommand{\d}{\delta}
\newcommand{\dist}{\mathrm{dist}}
\newcommand{\m}{\mathrm{mod}}

\title{Earthquakes in the length-spectrum Teichm\"uller spaces}


\date{\today}

\maketitle

\begin{abstract}
Let $X_0$ be a complete hyperbolic surface of infinite type that has a geodesic pants decomposition with cuff lengths bounded above. The length spectrum Teichm\"uller space $T_{ls}(X_0)$ consists of homotopy classes of hyperbolic metrics on $X_0$ such that the ratios of the corresponding simple closed geodesic for the hyperbolic metric on $X_0$ and for the other hyperbolic metric are bounded from the below away from $0$ and from the above away from $\infty$ (cf. \cite{ALPS}).
This paper studies earthquakes in the length spectrum Teichm\"uller space $T_{ls}(X_0)$. 
We find a necessary condition and several sufficient conditions on earthquake measure $\mu$ such that the corresponding earthquake $E^{\mu}$ describes the hyperbolic metric on $X_0$ which is in the length spectrum Teichm\"uller space. Moreover, we give examples of earthquake paths $t\mapsto E^{t\mu}$, for $t\geq 0$, such that $E^{t\mu}\in T_{ls}(X_0)$ for $0\leq t<t_0$, $E^{t_0\mu}\notin T_{ls}(X_0)$ and $E^{t\mu}\in T_{ls}(X_0)$ for $t>t_0$.
\end{abstract}

\section{Introduction}

The length spectrum Teichm\"uller space $T_{ls}(X_0)$ of a complete hyperbolic surface $X_0$ of infinite type consists of all homotopy classes of marked complete hyperbolic surfaces whose length spectrum is comparable to the length spectrum of the base point $X_0$ (cf. \cite{ALPS} and \S \ref{sec:2}). A reasonable assumption on $X_0$ (cf. \cite{ALPS}, \cite{Shiga}) is that there exists a geodesic pants decomposition $\mathcal{P} =\{\alpha_n\}$ and a constant $L_0>0$ such that
$$
l_{\alpha_n}(X_0)\leq L_0.
$$
Such a pants decomposition $\mathcal{P}$ is said to be {\it upper bounded} and from now on we assume that an upper bounded pants decomposition always exists.  An assignment of the lengths and the twists to the cuffs of $\mathcal{P}$ parametrizes the length spectrum Teichm\"uller space $T_{ls}(X_0)$ and the image of $T_{ls}(X_0)$ in these Fenchel-Nielsen coordinates is completely described (cf. \cite{ALPS}). Moreover, the Fenchel-Nielsen coordinates give a locally biLipschitz homeomorphism between $T_{ls}(X_0)$ and $l_{\infty}$ (cf. \cite{Sa3}). 

We study earthquakes in the length spectrum Teichm\"uller space $T_{ls}(X_0)$. Unlike many other deformations of hyperbolic structures on surfaces, earthquakes are completely described in terms of the initial hyperbolic structure (cf. Thurston \cite{Th1}). Thus one could say that earthquakes are natural in the context of the hyperbolic geometry. An earthquake is a deformation of a hyperbolic structure along the support of the earthquake (a geodesic lamination on the initial surface) by the amount given by the earthquake measure (a transverse measure to the earthquake support). Any two hyperbolic structures on a fixed surface are said to be {\it continuous} if the lift to the universal coverings of the identity map extends to a continuous map of the boundary circles (cf. Thurston \cite{Th1}). 
Thurston defined earthquakes and proved that any two continuous hyperbolic structures are related by a unique earthquake (cf. \cite{Th1}).  The upper bounded geodesic pants decomposition guarantees that the hyperbolic metric on $X_0$ and on the new marked hyperbolic surface are continuous and thus related by an earthquake. We study under which conditions on an earthquake measure $\mu$ the image hyperbolic structure $E^{\mu}(X_0)$ (under the corresponding earthquake $E^{\mu}$) is in the length spectrum Teichm\"uller space $T_{ls}(X_0)$.

Let $E^{\mu}:X_0\to X^{\mu}$ be an earthquake from the hyperbolic surface $X_0$ (which has an upper bounded pants decomposition) onto another hyperbolic surface $X^{\mu}$, where the measured geodesic lamination $\mu$ is the earthquake measure of $E^{\mu}$.  
Define the {\it length spectrum norm} of $\mu$ by
$$
\|\mu\|_{ls}=\sup_{\beta\in\mathcal{S}}\frac{\mu (\beta )}{l_{\beta}(X_0)},
$$
where $\mathcal{S}$ is the set of all simple closed geodesics on $X_0$ and $\mu (\beta )$ is the total $\mu$-mass deposited on $\beta\in\mathcal{S}$. 
A necessary condition for $E^{\mu}:X_0\to X^{\mu}$ to belong to the length spectrum Teichm\"uller space $T_{ls}(X_0)$ is that $\|\mu\|_{ls}<\infty$ (cf. \S \ref{sec:upper}). Such $\mu$ is said to be {\it length spectrum bounded}. However, it turns out that this condition is not sufficient (cf. \S \ref{sec:concl}). In fact, an earthquake path $t\mapsto (E^{t\mu}:X_0\to X^{t\mu})$, for $t\geq 0$, can start in $T_{ls}(X_0)$ leave $T_{ls}(X_0)$ at some finite time $t_0$ and return afterwards to $T_{ls}(X_0)$. Moreover, there exist length spectrum bounded earthquake measures $\mu$ such that the earthquake path $E^{t\mu}$ leaves and returns to $T_{ls}(X_0)$ infinitely many times (cf. \S \ref{sec:concl}).  

The only reason that $\|\mu\|_{ls}<\infty$ is not a sufficient condition for an earthquake $E^{\mu}:X_0\to X^{\mu}$ to belong to $T_{ls}(X_0)$ is that the lengths of simple closed curves might decrease too much under the earthquake thus making the ratio of lengths of the corresponding simple closed curves on $X_0$ and $X^{\mu}$ too small. We give sufficient conditions on $\mu$ in order to guarantee that $E^{\mu}$ is in $T_{ls}(X_0)$.

Let $\alpha_n$ be a cuff in the upper bounded geodesic pants decomposition of $X_0$ and let $P_n^1,P_n^2$ be the two pairs of pants in $\mathcal{P}$ with the common cuff $\alpha_n$. For a leaf $g$ of the support of $\mu$ that intersects $\alpha_n$, denote by $g_{comp}$ a component of $g\cap (P_n^1\cup P_n^2)$. 
The {\it winding number} $w_{\alpha_n}(g_{comp})$ of $g_{comp}$ around the curve 
$\alpha_n$ is defined in \S \ref{sec:lower-bound} and can be essentially thought of as the 
amount of winding of $g_{comp}$ around $\alpha_n$. The quantity $w_{\alpha_n}(g_{comp})$ 
becomes crucial when the angle between $\alpha_n$ and $g$ is larger than $\frac{\pi}{2}$. 
We have the following (cf. Theorem \ref{thm:bound-on-cuffs} and \S \ref{sec:upper})

\vskip .2 cm

\noindent {\bf Theorem 1.} {\it Let $X_0$ be a complete hyperbolic surface which has a geodesic pants decomposition $\mathcal{P} =\{\alpha_n\}_{n\in\mathbb{N}}$ such that
$$
l_{\alpha_n}(X_0)\leq L_0
$$
for some fixed $L_0>0$. Let $\mu$ be a measured (geodesic) lamination on $X_0$ such that
$$
\|\mu\|_{ls}<\infty .
$$
If there exist $C_0>1$, $C_0'>C(L_0,\|\mu\|_{ls})\geq 1$ for the constant $C(L_0,\|\mu\|_{ls})$ from Lemma \ref{lem:tr-lengh-earthquake-measure} and $C_1>0$ such that (for each cuff $\alpha_n$ of $\mathcal{P}$) $\mu$ satisfies one of the following:
\begin{enumerate}
\item $\mu (\alpha_n )>C_0l_{\alpha_n}(X_0)$
\item $\mu (\alpha_n )<\frac{1}{C_0'}l_{\alpha_n}(X_0)$
\item  the angle between $\alpha_n$ and a leaf $g$ of $\mu$ is less than or equal to $\frac{\pi}{2}$
\item $\frac{1}{C_0'}l_{\alpha_n}(X_0)\leq \mu (\alpha_n )\leq C_0l_{\alpha_n}(X_0)$, the angle between $\alpha_n$ and a leaf $g$ of $\mu$ is greater than $\frac{\pi}{2}$ and $w_{\alpha_n}(g_{comp})<C_1\frac{1}{l_{\alpha_n}(X_0)}$
\end{enumerate}
then the earthquake 
$$
E^{\mu}:X_0\to X^{\mu}
$$
belongs to the length spectrum Teichm\"uller space $T_{ls}(X_0)$.
}

\vskip .2 cm

The above theorem facilitates finding sufficient conditions on $\mu$ such that the whole earthquake path $t\mapsto E^{t\mu}$, for $t\geq 0$, stays in the length spectrum Teichm\"uller space $T_{ls}(X_0)$. We have (cf. Theorems \ref{thm:earthquake-paths} and \ref{thm:partition_of_cuffs})

\vskip .2 cm

\noindent
{\bf Theorem 2.} {\it
Let $X_0$ be a complete hyperbolic surface with an upper bounded geodesic pants decomposition $\mathcal{P}=\{\alpha_n\}$ and let $\mu$ be a measured geodesic lamination on $X_0$ with
$$
\|\mu\|_{ls}<\infty.
$$
Then $E^{t\mu}(X_0)=X^{t\mu}\in T_{ls}(X_0)$ for all $t\geq 0$ if there exists $C>0$ such that for each $\alpha_n$ one of the following holds:
\begin{enumerate}
\item the angle between $\alpha_n$ and a leaf of $\mu$ is less than $\frac{\pi}{2}$,
\item the angle between $\alpha_n$ and a leaf of $\mu$ is greater than $\frac{\pi}{2}$, and $w_{\alpha_n}(g_{comp})\leq C\frac{1}{l_{\alpha_n}(X_0)}$.
 \end{enumerate}
In addition, $E^{t\mu}(X_0)=X^{t\mu}\in T_{ls}(X_0)$ for all $t\geq 0$ if
$\mathcal{P}$ can be partitioned into $\mathcal{P}'$ and $\mathcal{P}''$ such that each $\alpha_n\in\mathcal{P}'$ satisfies either {\rm (1)} or {\rm (2)} for a fixed $C>0$,
and that for $\alpha_n\in\mathcal{P}''$
$$
l_{\alpha_n}(X_0)\to 0
$$
and
$$
\frac{\mu (\alpha_n)}{l_{\alpha_n}(X_0)}\to 0
$$
as $n\to\infty$.
}

\section{The Fenchel-Nielsen coordinates}
\label{sec:2}

Let $X_0$ be a complete hyperbolic surface without boundary of infinite type. Assume that there exists $L_0>0$ and a geodesic pants decomposition $\mathcal{P}=\{\alpha_n\}$ of $X_0$ such that for each $\alpha_n\in\mathcal{P}$
$$
l_{\alpha_n}(X_0)\leq L_0,
$$
where $l_{\alpha_n}(X_0)$ is the length of the geodesic representative of the curve $\alpha_n$ in the hyperbolic metric on $X_0$.

Consider the set of all homeomorphisms
$$
h:X_0\to X
$$
such that there exists $L>0$ with
$$
\frac{1}{L}\leq\frac{l_{\beta}(X)}{l_{\beta}(X_0)}\leq L
$$
for all simple closed geodesics $\beta\in\mathcal{S}$, where $l_{\beta}(X)$ is the length of the geodesic representative of $h(\beta )$ on $X$. The {\it length spectrum Teichm\"uller space} $T_{ls}(X_0)$ of the surface $X_0$ consists of all equivalence classes of the above homeomorphisms, where $(h,X)$ is equivalent to $(h',X')$ if there exists an isometry $I:X\to X'$ such that $(h')^{-1}\circ I\circ h:X_0\to X_0$ is homotopic to the identity with a bounded homotopy (cf. \cite{ALPS}).

The Fenchel-Nielsen coordinates on infinite type surfaces are defined in the same fashion as on the finite type surfaces (cf. \cite{ALPS}). 
We recall a characterization of the length spectrum Teichm\"uller space $T_{ls}(X_0)$ in terms of the Fenchel-Nielsen coordinates (cf. \cite{ALPS}, \cite{Sa3}):

\begin{theorem}
Let $X_0$  be an infinite type complete hyperbolic surface equipped with an upper bounded geodesic pants decomposition $\mathcal{P} =\{ \alpha_n\}_{n\in\mathbb{N}}$.  The normalized Fenchel-Nielsen coordinates
\begin{equation}
\label{eq:fn-norm}
F(X)=\Big{\{} \Big{(}\log \frac{l_{\alpha_n}(X)}{l_{\alpha_n}(X_0)},\frac{t_{\alpha_n}(X)-t_{\alpha_n}(X_0)}{\max\{ 1,|\log l_{\alpha_n}(X_0)|\}}\Big{)}\Big{\}}_{n\in\mathbb{N}}
\end{equation}
associated to each $X\in T_{ls}(X_0)$ induce a locally biLipschitz surjective homeomorphism
$$
F:T_{ls}(X_0)\to l^{\infty}.
$$
\end{theorem}

\section{The length spectrum bound on earthquake measures}

Let $X_0$ be a complete hyperbolic surface (without boundary) that has upper bounded geodesic pants decomposition $\mathcal{P}=\{\alpha_n\}$. Let $\mu$ be a measured lamination on $X_0$. Recall a definition of Thurston \cite{Th1}.

\begin{definition}
A measured lamination $\mu$ on $X_0$ is {\it Thurston bounded} is
$$
\|\mu\|_{Th}:=\sup_I\mu (I)<\infty
$$
where the supremum is over all geodesic arcs $I$ on $X_0$ of length $1$.
\end{definition}

We introduce a new notion of boundedness of measured laminations which is better suited for the length spectrum Teichm\"uller spaces.

\begin{definition}
A measured lamination $\mu$ is {\it length spectrum bounded} if
$$
\|\mu\|_{ls}:=\sup_{\beta}\frac{\mu (\beta )}{l_{\beta}(X_0)}<\infty
$$
where the supremum is over all simple closed geodesics $\beta$ on $X_0$.
\end{definition}

Consider a marking homeomorphism $h:X_0\to X$ for any $X\in T_{ls}(X_0)$. Then $h$ induces a homeomorphism $\tilde{h}$ of the boundaries of the universal coverings of $X_0$ and $X$. 
Note that $X_0$ is complete by the assumption. The surface $X$ has an upper bounded geodesic pants decomposition because $X\in T_{ls}(X_0)$ which implies that $X$ is a complete hyperbolic surface.
Since $X_0$ and $X$ are complete hyperbolic surfaces without the boundary the set of fixed points of the covering groups $G_0$ and $G$ are dense in the unit circle $S^1$. The map $\tilde{h}$ is first defined as a one to one correspondence between the fixed points of elements of $G_0$ and the fixed points of elements of $G$ which preserves the order on $S^1$. Therefore $\tilde{h}$ extends to a homeomorphism of $S^1$ which conjugates the action of $G_0$ onto the action of $G$.

Thurston \cite{Th1}
proved that each homeomorphism of the circle is realized as the continuous extension of an earthquake of the hyperbolic plane $\mathbb{H}^2$. Moreover if the homeomorphism of the circle is induced from a homeomorphism of $X_0$ to $X$ then the earthquake descends onto the surface $X_0$, namely the earthquake measure is invariant under the covering group of $X_0$ \cite{Th1}. 
An {\it earthquake measure} is a measured lamination $\mu$, and the earthquake $E^{\mu}:X_0\to X$ is a piecewise isometry on each stratum of $\mu$ (cf. \cite{Th1}).

Note that $t\mu$ for $t>0$ is also a measured lamination on $X_0$ which is simply obtained by scaling $\mu$ by a factor $t$. However, $E^{t\mu}$ is in general not an earthquake map even for $0<t<1$ \cite{GHL}. In fact, a map $E^{t\mu}:X_0\to E^{t\mu}(X_0)$ can be defined to be piecewise isometry on the strata of $t\mu$ and it turns out that its image $E^{t\mu}(X_0)$ does not have to be a complete surface. In this case, the lift of $E^{t\mu}$ to the universal coverings does not extend to a homeomorphism of the boundary circles.

When $\mu$ is Thurston bounded then the map $E^{\mu}:X_0\to E^{\mu}(X_0)$ is surjective and the image $E^{\mu}(X_0)=X_{\mu}$  is a complete hyperbolic surface that is quasiconformally equivalent to $X_0$ (cf. \cite{Th1}, \cite{GHL}, \cite{EMM}, \cite{Sa1}). Since $\| t\mu\|_{Th}=t\|\mu\|_{Th}<\infty$ if $\|\mu\|_{Th}<\infty$, it follows that the earthquake path $E^{t\mu}:X_0\to X^{\mu}$ is well-defined for all $t\geq 0$. Thus in $T_{qc}(X_0)$ every point is connected by an earthquake path to the base point $X_0$ and the topology of $T_{qc}(X_0)$ can be recovered from the uniform weak* topology on the space of earthquake measures on $X_0$ \cite{MiySa}. We consider the question whether in $T_{ls}(X_0)$ every point is connected to the base point by an earthquake path.

\section{A lower bound on the lengths of the cuffs}
\label{sec:lower-bound}

In this section we estimate the ratio $l_{\alpha_n}(X_0)/l_{\alpha_n}(X_{\mu})$. To achieve this, we prove a few preparatory lemmas. Let $L_0>0$ be such that 
$$
l_{\alpha_n}(X_0)\leq L_0.
$$

Let $C>0$ such that $\|\mu\|_{ls}<C$. Then
$$
\mu (\alpha_n)\leq Cl_{\alpha_n}(X_0)
$$
by the definition of the length spectrum norm on the measured laminations.

Let $P_1$ and $P_2$ be two geodesic pairs of pants with cuff lengths bounded by $L_0>0$ that are glued along a common cuff $\alpha$. Let $g$ be either a simple closed geodesic in $P_1\cup P_2$ or a simple geodesic arc in $P_1\cup P_2$ which is joining two boundary cuffs of $P_1\cup P_2$ and transversely intersecting $\alpha$. We define the {\it winding number} $w_{\alpha}(g)$ of the arc $g$ around the curve $\alpha$ as follows. Let $\gamma_i$, for $i=1,2$, denote the unique simple geodesic arc in $P_i$ which is orthogonal to $\alpha$ at both of its endpoints. Divide each $P_i$ into two right angled hexagons $\Sigma_i^j$ for $j=1,2$ by cutting along three geodesic arcs  orthogonal to pairs of cuffs of $P_i$. Let 
$$\gamma_i^j:=\gamma_i\cap \Sigma_i^j
$$
for $i,j=1,2$.  Note that the length of $\gamma_i^j$ is half the length of $\gamma_i$.

Let $g_s$, for $s=1,2,\ldots ,k$, be the components of the intersections $g\cap P_i$ for $i=1,2$.
If the angle from $\alpha_n$ to $g_s$ is greater than $\frac{\pi}{2}$, we define the {\it winding number} $w_{\alpha}(g)$ by 
$$
w_{\alpha}(g)=\max_{1\leq s\leq k;\ i,j=1,2} \#(g_s,\gamma_i^j)
$$
where $\#(g_s,\gamma_i^j)$ is the number of intersections of $g_s$ and $\gamma_i^j=\gamma_i\cap\Sigma_i^j$. Note that $w_{\alpha}(g)-\#(g_s,\gamma_i^j)\leq 2$ for all $s=1,2,\ldots ,k$.

In order to use the winding number, we need the following lemma.

\begin{lemma}
\label{lem:twisting_into_radius}
Consider the universal covering $\pi:\mathbb{H}^2\to X_0$ such that one lift $\tilde{\alpha}_n$ is the positive $y$-axis.
Let $\tilde{g}$ be a lift to $\mathbb{H}^2$ of a leaf $g$ of the measured lamination 
$\mu$ that intersects the positive $y$-axis between $i$ and $e^{l_{\alpha_n}(X_0)}i$. Denote by $k_1<0$ and $k_2>0$ the endpoints of $\tilde{g}$. Then
$$
1\leq -k_1k_2\leq e^{2L_0}.
$$
If the angle between $\alpha_n$ and $g_{comp}$ is greater than $\frac{\pi}{2}$, then
$$
-k_1\geq C(L_0)e^{-l_{\alpha_n}(X_0)w_{\alpha_n}(g_{comp})},
$$
where $g_{comp}$ is a component of $g\cap (P_1\cup P_2)$ and $P_1,P_2$ are two pairs of pants in $\mathcal{P}$ with a common cuff $\alpha_n$.
\end{lemma}

\begin{proof}
Since $iy$ with $1\leq y\leq e^{l_{\alpha_n}(X_0)}\leq e^{L_0}$ belongs to the geodesic with endpoints $k_1<0$ and $k_2>0$, we have
$$
|iy-\frac{k_2+k_1}{2}|=\frac{k_2-k_1}{2}.
$$
This gives the first inequality in the above lemma.

Let $d>0$ be the length of the orthogonal to $\alpha_n$ and the side of the hexagon $\Sigma_i^j$ opposite $\alpha_n$, for fixed $i,j$. 
Let $\varphi>0$ be such that the distance $d$ between the positive $y$-axis and the Euclidean half-line through the origin which subtends the angle $\varphi$ with the $x$-axis satisfies
$$
\sin\varphi =\frac{1}{\cosh d}.
$$
Let $r>0$ be such that the geodesic with endpoints $-k_1e^{-r}$ and $k_2e^{-r}$ goes through the point $e^{i\varphi}$ of the unit circle centered at $0$. It follows that
$$\frac{r}{l_{\alpha_n}(X_0)}\leq w_{\alpha_n}(g_{comp})+2
$$
because $\frac{r}{l_{\alpha_n}(X_0)}$ is the number of translates of $\tilde{g}$ (under the covering transformation corresponding to $\alpha_n$) between $\tilde{g}$ and the geodesic with endpoints $-k_1e^{-r}$ and $k_2e^{-r}$, each translate intersects the lift of $\gamma_i^j$ (adjacent to the $y$-axis) exactly once, and these translates glued together form a single component covering $g_{comp}$ in $P_i$. 

Using Euclidean geometry, we have
$$
|e^{i\varphi}-\frac{k_1+k_2}{2}e^{-r}|^2=(\frac{k_2-k_1}{2})^2e^{-2r}
$$
which implies
$$
k_2\leq C'(L_0)e^r\leq C''(L_0)e^{l_{\alpha_n}(X_0)(
w_{\alpha_n}(g_{comp})+2)}.
$$
Then the first inequality in the theorem implies the second inequality.
\end{proof}

\vskip .2 cm

Let $\tilde{\mu}$ be the lift of $\mu$ to the universal covering $\mathbb{H}^2$
and let $E^{\tilde{\mu}}:\mathbb{H}^2\to\mathbb{H}^2$ be the corresponding earthquake. Let $O$ be the stratum of $\tilde{\mu}$ that contains $e^{l_{\alpha_n}(X_0)}i$ and normalize the earthquake such that $E^{\tilde{\mu}}|_O=id$. Let $O_1$ be the stratum  of $\tilde{\mu}$ which contains $i$ that is the image of $O$ under the covering map $B\in PSL_2(\mathbb{R})$ of the geodesic $\alpha_n$. Note that $B(z)=e^{-l_{\alpha_n}(X_0)}z$. Let $B^{\tilde{\mu}}$ be the covering map for $\alpha_n$ on the surface $E^{\mu}(X_0)=X^{\mu}$. Then by \cite{EpsteinMarden} we have
$$
B^{\tilde{\mu}}=E^{\tilde{\mu}}|_{O_1}\circ B.
$$

Since $E^{\tilde{\mu}}$ is a left earthquake, it follows that
$E^{\tilde{\mu}}|_{O_1}$ is a hyperbolic translation with the axis separating $O$ from $O_1$. Let $k_1<0$ and $k_2>0$ be the repelling and the attracting fixed points of $E^{\tilde{\mu}}|_{O_1}$ and let $m\geq 0$ be its translation length. Then we have
\begin{equation}
\label{eq:trace_under_earthq}
\begin{split}
trace(B^{\tilde{\mu}})=\frac{e^mk_2-k_1}{e^{m/2}(k_2-k_1)}e^{-\frac{l}{2}}+ \frac{k_2-e^mk_1}{e^{m/2}(k_2-k_1)}e^{\frac{l}{2}}\\ =2\cosh\frac{m-l}{2}-\frac{2k_1}{k_2-k_1}[\cosh\frac{m+l}{2}-cosh\frac{m-l}{2}]
\end{split}
\end{equation}
where for short $l=l_{\alpha_n}(X_0)$.
The above equation gives the inequality
\begin{equation}
\label{eq:trace_lower_estimate}
trace(B^{\tilde{\mu}})\geq 2-\frac{2k_1}{k_2-k_1}ml.
\end{equation}

\vskip .2 cm

We relate the translation length $m$ to the earthquake measure $\mu (\alpha_n)$ using the following lemma.

\begin{lemma}
\label{lem:tr-lengh-earthquake-measure}
Let $d>0$ be the length of a geodesic arc $I$ in $\mathbb{H}^2$ that transversely intersects a geodesic lamination $\mu$. Denote by $E^{\mu}$ the left earthquake of $\mathbb{H}^2$ which is normalized to be the identity on stratum $O$ of $\mu$ that contains an endpoint of $I$ and denote by $O_1$ another stratum of $\mu$ that contains the other endpoint of $I$. Then there exists $C(d,\mu (I))\geq 1$ such that
$$
\mu (I)\leq m\leq C(d,\mu (I))\mu (I),
$$
where $m$ is the translation length of $E^{\mu}|_{O_1}$.
\end{lemma} 

\begin{proof}
Let $S$ and $T$ be two hyperbolic translations whose axes are disjoint and both intersect a closed arc of length $d$. Denote by $\tau (S)$ the translation length of $S$. Then (cf. \cite{Th1}, \cite{GHL})
$$
\tau (S)+\tau (T)\leq\tau (S\circ T)\leq\tau (S)+\tau (T)+C'(d)\min\{\tau (S),\tau (T)\} d^2
$$
for some constant $C'(d)>0$.

Assume that only finitely many leaves $\{ g_1,\ldots ,g_n\}$ of $\mu$ intersect the geodesic arc $I$. Let $T_1,\ldots ,T_n$ be the hyperbolic translations whose axes are $g_1,\ldots ,g_n$ and whose translation lengths are $\mu (g_1),\ldots ,\mu (g_n)$. The above inequality gives
$$
\sum_{i=1}^n\tau (T_i)\leq \tau (T_1\circ\cdots\circ T_n)\leq \sum_{i=1}^n\tau (T_i) +C'(d)\sum_{i=1}^n\tau (T_i)d^2.
$$
Since $\tau (T_i)=\mu (g_i)$, $\sum_{i=1}^n\tau (g_i)=\mu (I)$ and $m=\tau (T_1\circ\cdots\circ T_n)$ the lemma is proved in this case. For the general $\mu$, note that $E^{\mu}$ is approximated by finite earthquakes and the earthquake measure $\mu$ is approximated by the measure of finite earthquakes. The lemma follows by continuity.
\end{proof}

Assume that there exists $C_0>1$ such that $$\mu (\alpha_n)\geq C_0l,$$ where for short $l=l_{\alpha_n}(X_0)$. By Lemma \ref{lem:tr-lengh-earthquake-measure} we have that
$$
m\geq C_0l.
$$ 
Then the equation (\ref{eq:trace_under_earthq}) 
gives
\begin{equation*}
trace(B^{\tilde{\mu}})\geq 2+\Big{(}\frac{C_0-1}{2}\Big{)}^2l^2
\end{equation*}
which implies that 
\begin{equation}
\label{eq:lower_est_lower_bound_on_measure}
l_{\alpha_n}(X^{\mu})\geq C'(L_0)l_{\alpha_n}(X_0)
\end{equation}
for some $C'(L_0)>0$.

\vskip .2 cm

If there exists $C_0'>C(L_0,\|\mu\|_{ls})\geq 1$ for the constant $C(L_0,\|\mu\|_{ls})$ from Lemma \ref{lem:tr-lengh-earthquake-measure} such that
$$
\mu (\alpha_n)\leq \frac{1}{C_0'}l_{\alpha_n}(X_0)
$$
then (by Lemma \ref{lem:tr-lengh-earthquake-measure} again) we have
$$m\leq C_1(L_0,\|\mu\|_{ls})l_{\alpha_n}(X_0)$$
where $C_1(L_0,\|\mu\|_{ls})=\frac{C(L_0,\|\mu\|_{ls})}{C_0'}<1$.
Then the equation (\ref{eq:trace_under_earthq}) 
gives
\begin{equation*}
trace(B^{\tilde{\mu}})\geq 2+\Big{(}\frac{1-1/C_1}{2}\Big{)}^2l^2
\end{equation*}
which implies that 
\begin{equation}
\label{eq:lower_est_upper_bound_on_measure}
l_{\alpha_n}(X^{\mu})\geq C''(L_0,\|\mu\|_{ls})l_{\alpha_n}(X_0)
\end{equation}
for some $C''(L_0,\|\mu\|_{ls})>0$.

\vskip .2 cm

Assume that
\begin{equation}
\label{eq:squeezed_measure}
\frac{1}{C_0'}l_{\alpha_n}(X_0)<
\mu (\alpha_n)<C_0l_{\alpha_n}(X_0)
\end{equation}
which implies 
$$
\frac{1}{C_0'}l_{\alpha_n}(X_0)<
m<C_0C(L_0,\|\mu\|_{ls})l_{\alpha_n}(X_0).
$$

Then the inequality (\ref{eq:trace_lower_estimate}) gives
\begin{equation}
\label{eq:lower_estimate_m_close_to_l}
trace(B^{\tilde{\mu}})\geq 2-\frac{1}{C_2(L_0,\|\mu\|_{ls})}\frac{2k_1}{k_2-k_1}l^2
\end{equation}
which implies that $\frac{l_{\alpha_n}(X_0)}{l_{\alpha_n}(X^{\mu})}\leq C'''$ when the angle between $\alpha_n$ and $g_{comp}$ is less than $\frac{\pi}{2}$.
If the angle between $\alpha_n$ and $g_{comp}$ is greater than $\frac{\pi}{2}$,
then (\ref{eq:lower_estimate_m_close_to_l}) together with Lemma \ref{lem:twisting_into_radius} implies
\begin{equation}
\label{eq:lenght_m_close_to_l}
l_{\alpha_n}(X^{\mu})\geq C_2(L_0,\|\mu\|_{ls})e^{-l_{\alpha_n}(X_0)w_{\alpha_n} (g_{comp})} l_{\alpha_n}(X_0)
\end{equation}
where $g$ is a leaf of $\mu$ which intersects $\alpha_n$ and $g_{comp}$ is a component 
of $g\cap P_i$ for either $i=1$ or $i=2$. To estimate 
$\frac{l_{\alpha_n}(X_0)}{l_{\alpha_n}(X^{\mu})}$ we need to estimate the right hand side of (\ref{eq:lenght_m_close_to_l}).

\begin{lemma}
\label{lem:bound_on_twists}
Let $\alpha_n$ be a simple closed geodesic on $X_0$ from the fixed geodesic pants decomposition $\mathcal{P}$ and let $P_1,P_2$ be the two (possibly equal) pairs of pants in the decomposition $\mathcal{P}$ with a common cuff $\alpha_n$.
Let $w_{\alpha_n}(g_{comp})$ be a twisting number around $\alpha_n$ of a component $g_{comp}$ of $g\cap P_i$ for a leaf $g$ of the measured lamination $\mu$. 
Then
$$
w_{\alpha_n}(g_{comp})\leq C \frac{\max\{ 1,|\log l_{\alpha_n}(X_0)|\}}{\mu (\alpha_n)},
$$
where $C=C(\|\mu\|_{ls})>0$ depends on the length spectrum norm $\|\mu\|_{ls}$ of $\mu$.
\end{lemma}

\begin{proof}
We consider the leaves of $\mu$ that intersect $\alpha_n$. For each such leaf $g$, we divide it into components of $g\cap P_i$. Observe that $w_{\alpha_n}(g_{comp})$ and $w_{\alpha_n}(g'_{comp})$ for any two components $g_{comp}$ and $g'_{comp}$ differ by at most an additive constant which  can be taken to be $2$. Then it follows that the number of intersections between each component and the arc $\gamma_n^i$ is up to an additive constant equal to $2w_{\alpha_n}(g_{comp})$ for any component $g_{comp}$. Then we have
$$
\mu (\alpha_n) w_{\alpha_n}(g_{comp})\leq \mu (\beta_n )\leq \|\mu\|_{ls} l_{\beta_n}(X_0)
$$ 
and since
$$
l_{\beta_n}(X_0)\leq C'\max\{ 1,l_{\alpha_n}(X_0)\} 
$$
we obtain the desired conclusion.
\end{proof}

Thus if $\mu$ satisfies (\ref{eq:squeezed_measure}) and if the angle between $\alpha_n$ and $g_{comp}$ is greater than $\frac{\pi}{2}$, Lemma \ref{lem:bound_on_twists}
gives
$$
w_{\alpha_n}(g_{comp})\leq C(L_0,\|\mu\|_{ls})\frac{|\log l_{\alpha_n}(X_0)|}{l_{\alpha_n}(X_0)}.
$$

If the angle between $\alpha_n$ and $g_{comp}$ is greater than $\frac{\pi}{2}$ and if
there exists $C>0$ such that
$$
w_{\alpha_n}(g_{comp})\leq C\frac{1}{l_{\alpha_n}(X_0)}
$$
then 
$$
\frac{l_{\alpha_n}(X_0)}{l_{\alpha_n}(X^{\mu})}\leq C'
$$
for some $C'$.

\vskip .2 cm

To summarize, we have 

\begin{theorem}
\label{thm:bound-on-cuffs}
Let $X_0$ be a complete hyperbolic surface which has a geodesic pants decomposition $\mathcal{P} =\{\alpha_n\}_{n\in\mathbb{N}}$ such that
$$
l_{\alpha_n}(X_0)\leq L_0
$$
for some fixed $L_0>0$. Let $\mu$ be a measured (geodesic) lamination on $X_0$ such that
$$
\|\mu\|_{ls}<\infty .
$$
If there exist $C_0>1$, $C_0'>C(L_0,\|\mu\|_{ls})\geq 1$ for the constant $C(L_0,\|\mu\|_{ls})$ from Lemma \ref{lem:tr-lengh-earthquake-measure} and $C_1>0$ such that (for each cuff $\alpha_n$ of $\mathcal{P}$) $\mu$ satisfies one of the following:
\begin{enumerate}
\item $\mu (\alpha_n )>C_0l_{\alpha_n}(X_0)$
\item $\mu (\alpha_n )<\frac{1}{C_0'}l_{\alpha_n}(X_0)$
\item  the angle between $\alpha_n$ and a leaf $g$ of $\mu$ is less than or equal to $\frac{\pi}{2}$
\item $\frac{1}{C_0'}l_{\alpha_n}(X_0)\leq \mu (\alpha_n )\leq C_0l_{\alpha_n}(X_0)$, the angle between $\alpha_n$ and a leaf $g$ of $\mu$ is greater than $\frac{\pi}{2}$ and $w_{\alpha_n}(g_{comp})<C_1\frac{1}{l_{\alpha_n}(X_0)}$
\end{enumerate}
then there exists $C^{*}=C^{*}(L_0,\|\mu\|_{ls},C_1)>0$ such that
$$
\frac{l_{\alpha_n}(X_0)}{l_{\alpha_n}(E^{\mu}(X_0))}\leq C^{*}
$$
for all $\alpha_n\in\mathcal{P}$.
\end{theorem}

\section{Upper bounds on lengths}
\label{sec:upper}

We shortly describe the upper bounds on the lengths of simple closed geodesics under the earthquake map. Namely, if the support of earthquake $E^{\mu}$ which intersects $\alpha_n$ consists of finitely many closed geodesics then it is standard that
$$
l_{\alpha_n}(X^{\mu})\leq l_{\alpha_n}(X_0)+\mu (\alpha_n).
$$
Indeed, the proof is by lifting the earthquake to the universal covering and noting that the shear is always to the left (cf. \cite{Ker}). Since the cocycle map for $E^{\mu}$ is obtained by approximations with finitely many leaves whose total measure is $\mu (\alpha_n)$ (cf. \cite{EpsteinMarden}), we obtain the above inequality for arbitrary earthquakes.

Since $\mu (\alpha_n)\leq \|\mu\|_{ls}l_{\alpha_n}(X_0)$, the above inequality implies
$$
\log\frac{l_{\alpha_n}(X^{\mu})}{l_{\alpha_n}(X_0)}\leq\log (1+\frac{\mu (\alpha_n)}{l_{\alpha_n}(X_0)})\leq \frac{\mu (\alpha_n)}{l_{\alpha_n}(X_0)}\leq\|\mu\|_{ls}.
$$

\section{Bounding the twists of the cuffs}

In this section we bound the twists $|t_{\alpha_n}(E^{\mu}(X_0))-t_{\alpha_n}(X_0)|$. We recall that $t_{\alpha_n}(X_0)$ is chosen such that $0\leq t_{\alpha_n}(X_0)<l_{\alpha_n}(X_0)$. By the proof in \cite[Theorem 2.1, Step I, Case 1]{Sa3}, we have that (cf. Figure 2 in \cite{Sa3}) 
$$
|t_{\alpha_n}(X^{\mu})|\leq \max\{1,|\log l_{\alpha_n}(X^{\mu})|\}+l_{\beta_n}(X^{\mu}).
$$
First of all $l_{\alpha_n}(X^{\mu})$ is proportional to $l_{\alpha_n}(X_0)$ with the universal constants depending on $\|\mu\|_{ls}$ by the previous two sections. Moreover, we have that
\begin{equation}
\label{eq:bound_on_tw}
l_{\beta_n}(X^{\mu})\leq l_{\beta_n}(X_0)+\mu (\beta_n ) .
\end{equation}
Since $l_{\beta_n}(X_0)\leq C\max\{ 1,|\log l_{\alpha_n}(X_0)|\}$ and
$\mu (\beta_n)\leq\|\mu\|_{ls} l_{\beta_n}(X_0)$, we obtain
$$
l_{\beta_n}(X^{\mu})\leq C(\|\mu\|_{ls})\max\{ 1,|\log l_{\alpha_n}(X_0)|\}
$$
which proves the desired bound on $|t_{\alpha_n}(E^{\mu}(X_0))-t_{\alpha_n}(X_0)|$. 

To see that (\ref{eq:bound_on_tw}) holds, it is enough to note that it holds when the support of $\mu$ intersects $\beta_n$ in finitely many leaves and the general case follows by approximations of earthquakes cocycles with cocycles supported on finitely many leaves (cf. \cite{EpsteinMarden}).

\section{Necessity of the condition $\|\mu\|_{ls}<\infty$}

We show that $$\|\mu\|_{ls}<\infty$$ is necessary for $E^{\mu}(X_0)=X^{\mu}$ to satisfy
$$d_{ls}(X_0,X^{\mu})<\infty.$$
Assume on the contrary that there exists a sequence $\beta_{n}$ of the simple closed geodesics on $X_0$ such that $$\frac{\mu (\beta_n)}{l_{\beta_n}(X_0)}\to\infty$$ as $n\to\infty$. Then after normalizing the earthquake $E^{\tilde{\mu}}:\mathbb{H}^2\to\mathbb{H}^2$ as in \S \ref{sec:lower-bound}, we get by (\ref{eq:trace_under_earthq}) that
$$
trace(B^{\mu})\geq 2\cosh \frac{\mu ( \beta_n)-l_{\beta_n}(X_0)}{2}
$$
which implies 
$$
l_{\beta_n}(X^{\mu})\geq C(\mu (\beta_n)-l_{\beta_n}(X_0)).
$$
In conclusion
$$
\frac{l_{\beta_n}(X^{\mu})}{l_{\beta_n}(X_0)}\geq C\frac{\mu (\beta_n)}{l_{\beta_n}(X_0)}-C\to\infty
$$
as $n\to\infty$ which contradicts $d_{ls}(X_0,X^{\mu})<\infty$. Thus $\|\mu\|_{ls}<\infty$ is a necessary condition.

\section{Conclusions}
\label{sec:concl}

We established that the condition $\|\mu\|_{ls}<\infty$ is necessary for $E^{\mu}(X_0)=X^{\mu}$ to satisfy $d_{ls}(X_0,X^{\mu})<\infty$. From Theorem \ref{thm:bound-on-cuffs},  we immediately conclude that the following earthquake paths are inside $T_{ls}(X_0)$.

\begin{theorem}
\label{thm:earthquake-paths}
Let $X_0$ be a complete hyperbolic surface with an upper bounded geodesic pants decomposition $\mathcal{P}=\{\alpha_n\}$ and let $\mu$ be a measured geodesic lamination on $X_0$ with
$$
\|\mu\|_{ls}<\infty.
$$
Then $E^{t\mu}(X_0)=X^{t\mu}\in T_{ls}(X_0)$ for all $t\geq 0$ if there exists $C>0$ such that for each $\alpha_n$ one of the following holds:
\begin{enumerate}
\item the angle between $\alpha_n$ and a leaf of $\mu$ is less than $\frac{\pi}{2}$,
\item the angle between $\alpha_n$ and a leaf of $\mu$ is greater than $\frac{\pi}{2}$, and $w_{\alpha_n}(g_{comp})\leq C\frac{1}{l_{\alpha_n}(X_0)}$.
 \end{enumerate}
\end{theorem}

If there is no $C>0$ such that $w_{\alpha_n}(g_{comp})\leq C\frac{1}{l_{\alpha_n}(X_0)}$, it is still possible that the whole earthquake path remains in $T_{ls}(X_0)$. 

\begin{theorem}
\label{thm:partition_of_cuffs}
Assume that $\mathcal{P}$ is partitioned into $\mathcal{P}'$ and $\mathcal{P}''$ such that there exists $C>0$ with
$$
w_{\alpha_n}(g_{comp})\leq \frac{C}{l_{\alpha_n}(X_0)}
$$
for all $\alpha_n\in\mathcal{P}'$,
and that for $\alpha_n\in\mathcal{P}''$
$$
l_{\alpha_n}(X_0)\to 0
$$
and
$$
\mu (\alpha_n)=o(l_{\alpha_n}(X_0))
$$
as $n\to\infty$. Then the earthquake path $t\mapsto E^{t\mu}(X_0)=X^{t\mu}$ is contained in $T_{ls}(X_0)$ for all $t\geq 0$.
\end{theorem}

\begin{proof}
Since $\mu (\alpha_n)=o(l_{\alpha_n}(X_0))$, it follows that there exists $n_0=n_0(t)$ such that $t\mu (\alpha_n)<\frac{1}{2}l_{\alpha_n}(X_0)$ for $n\geq n_0$. The conclusion follows by the proof of Theorem \ref{thm:bound-on-cuffs}. 
\end{proof}

\vskip .2 cm

Finally, we establish that not all earthquake paths $t\mapsto E^{t\mu}(X_0)$ whose earthquake measure satisfy $\|\mu\|_{ls}<{\infty}$ remain in $T_{ls}(X_0)$ for all $t>0$.  We define $\mu$ by choosing the support to be a family of 
simple closed curves $\{\beta_k\}_{k=1}^{\infty}$ such that each $\beta_k$ is contained in the union $P_k^1\cup P_k^2$ of two pairs of pants $P_k^1,P_k^2\in\mathcal{P}$, that $\beta_k$ intersects the common cuff $\alpha_k$ of $P_k^1,P_k^2$ in a single point, that the angle between $\alpha_k$ and $\beta_k$ is greater than $\frac{\pi}{2}$, and that
$$
w_{\alpha_k}(\beta_k)l_{\alpha_k}(X_0)\to\infty
$$
as $k\to\infty$. We choose the measure $\mu$ such that
$$
\mu (\alpha_k)=\frac{1}{2}l_{\alpha_k}(X_0)
$$
or equivalently that
the Dirac measure on $\{\beta_k\}_{k=1}^{\infty}$ multiplied by $\frac{1}{2}l_{\alpha_k}(X_0)$. 

Note that $t\mu (\alpha_k)=\frac{t}{2}l_{\alpha_k}(X_0)$ which implies that $X^{t\mu}\in T_{ls}(X_0)$ for $t\geq 0$, $t\neq 2$. 
By (\ref{eq:trace_under_earthq}), we have that
$$
l_{\alpha_k}(X^{2\mu})\leq Ce^{-w_{\alpha_k}(\beta_k)l_{\alpha_k}(X_0)}l_{\alpha_k}(X_0)
$$
which implies that $X^{2\mu}=E^{2\mu}(X_0)\notin T_{ls}(X_0)$. 
We note that $X^{2\mu}$ is a hyperbolic surface homeomorphic to $X_0$ where no pinching occurred.

\begin{remark}We can hypothetically think of $X^{2\mu}$ as belonging to some ``augmentation'' of $T_{ls}(X_0)$ in contrast to the augmentation of Teichm\"uller spaces of finite surfaces where only pinched surfaces appear. This idea will be pursued elsewhere.
\end{remark}

Note that this behavior that an earthquake path leaves $T_{ls}(X_0)$ at some time $t$ and returns in $T_{ls}(X_0)$ afterwards can be repeated for infinitely many values of $t$ by simply choosing different multiples of the Dirac measures along different subsequences of $\beta_k$. In particular, an earthquake path $t\mapsto E^{t\mu}(X_0)$ can leave and return to $T_{ls}(X_0)$ at infinitely many points $t$.

\end{document}